\pdfoutput=1
\documentclass[12pt,a4paper,oneside]{amsart}

\usepackage{lmodern}
\usepackage[T1]{fontenc}
\usepackage[latin1]{inputenc}
\usepackage[left=3cm, right=3cm, bottom=4cm, top=3cm]{geometry}
\usepackage{amsmath, amssymb, amsfonts, graphics, graphicx, float}
\usepackage{multirow}
\usepackage{bigdelim}
\usepackage{parskip}
\usepackage{enumerate}
\usepackage{csquotes}
\usepackage{pgf,tikz}
\usetikzlibrary{arrows}

\newtheorem{thm}{Theorem}
\newtheorem{prop}[thm]{Proposition}

\newtheorem{cor}[thm]{Corollary}
\newtheorem{definition}[thm]{Definition}
\newtheorem{remark}[thm]{Remark}
\theoremstyle{definition}
\newtheorem{ex}[thm]{Example}
\newtheorem{note}[thm]{Notation}
\numberwithin{equation}{section}

\usepackage{multirow}
\usepackage{bigdelim}

\newcommand{\N}{\mathbb{N}}

\newcommand{\R}{\mathbb{R}}
\newcommand{\C}{\mathbb{C}}
\newcommand{\Q}{\mathbb{Q}}
\newcommand{\rr}{\R[V^k]}
\newcommand{\rx}{\R[V^k]^{S_n}}

\newcommand\restr[2]{{
  \left.\kern-\nulldelimiterspace 
  #1
  \right|_{#2} 
  }}

\DeclareMathOperator{\sym}{sym}

\usepackage{hyperref} 

\title{Deciding positivity  of multisymmetric polynomials}
\author{Paul G\"orlach}
\address{Mathematisches Institut der Universit\"at Bonn\\
Endenicher Allee 60\\
53115 Bonn, Germany }
\email{goerlach@uni-bonn.de}
\author{Cordian  Riener}
\address{Aalto Science Institute, Aalto University\\
PO Box  11000\\
FI-00076 Aalto, Finland}
\email{cordian.riener@aalto.fi}
\author{Tillmann Wei\ss er}
\address{Fachbereich Mathematik und Statistik,
Universit\"at Konstanz, 78457 Konstanz, Germany.}

\email{tillmann.weisser@uni-konstanz.de}
\begin{document}
\begin{abstract}
The question how to certify non-negativity of a polynomial function lies at the heart of Real Algebra and also has important applications to Optimization. In this article we investigate the question of non-negativity in the context of multisymmetric  polynomials. In this setting we generalize the characterization of non-negative symmetric polynomials given in \cite{timofte-2003,Rie} by adapting the method of proof developed in \cite{Rie2}. One particular case where  our results can be applied is the question of certifying that a (multi-)symmetric polynomial defines a convex function. As a direct corollary of our main result we deduce that in the case of a fixed degree it is possible to derive a method to test for convexity which makes use of the special structure of (multi-) symmetric polynomials. In particular it follows that we are able to drastically simplify the algorithmic complexity of this question in the presence of symmetry. This is not to be expected in the general (i.e. non-symmetric) case,  where it is known that testing for convexity is NP-hard already in the case of polynomials of degree $4$\cite{amir}.

\end{abstract}

\maketitle

\section{Introduction}
A real polynomial  is called positive (non-negative) if its evaluation on every real point  is positive  (non-negative). The study of this property of polynomials functions is one of the aspects that separates Real Algebraic Geometry from Algebraic Geometry over algebraically closed fields. Indeed, Real Algebraic Geometry  developed building on Hilbert's problem of characterizing non-negative polynomials via sums of squares. On the complexity side, it is know that the problem to algorithmically decide whether a given polynomial assumes only positive or non-negative values is NP-hard in general (see for example \cite{murty,blum}) and is essential for example to understand global optimization of polynomial functions. Besides the general results some authors have studied particular cases of polynomials which are invariant under group actions, for example by permuting the variables. In particular, symmetric polynomials, i.e. polynomials invariant under all permutations of the variables, exhibit some interesting properties that behave differently over real closed and algebraically closed fields.  For example, in \cite{BS} the authors show, that the number of connected components of the orbit space of a complex variety which is defined by symmetric polynomials of a given maximal degree  is bounded by a quantity which (for a large number of variables) only depends on this maximal degree. In contrast to that there are examples of real varieties where this does not hold, i.e., this quantity  actually grows with the number of variables. Therefore, the geometry of symmetric real varieties and semi-algebraic sets can be much more complicated. 

Combining  Artin's solution to Hilbert's 17th problem with Hermite's quadratic form, which characterizes real univariate polynomials with only real roots, Procesi \cite{procesi} was able to give a {\em Positivstellensatz} characterizing all symmetric non-negative polynomials. However, this characterization is not very advantageous in situations, where the degree of the polynomials is much smaller than the number of variables.  For this situations, Timofte \cite{timofte-2003} was able to provide a characterization of symmetric non-negative functions of a fixed degree that can be used to algorithmically certify non-negativity: He could establish that a symmetric polynomial  of degree $2d$ is non-negative if and only if it is non-negative on all points  with at most $d$ distinct coordinates.  This observation, which generalized an earlier statement by Harris \cite{harris1999real}, leads to a number of interesting consequences. For example, it provides an essential part to the description of the asymptotical behavior of the cone of symmetric non-negative forms of a given degree when the number of variables  grows \cite{GC}.  Algorithmically, this result allows to show that the complexity of deciding non-negativity  of a symmetric function with a fixed degree only grows  polynomially in the number of variables. Following the work of Timofte, the second author was able to provide short proofs of this characterization of symmetric non-negative polynomials \cite{Rie,Rie2}.

{\bf Contributions:} In this article we extend the previous results on symmetric polynomials to arrive at a similar characterization of multisymmetric polynomial functions that assume only positive (non-negative) values.  The class of  multisymmetric polynomials naturally generalizes the symmetric polynomials. Whereas symmetric polynomials are invariant by all permutations of the variables, multisymmetric polynomials  can be thought of as functions which are invariant  under simultaneously permuting $k$-tuples of variables. Similar to the case of symmetric polynomials, we are able to show that when the degree of such a polynomial is sufficiently smaller than  the number of variables, also for these multisymmetric polynomials non-negativity can be checked on a lower-dimensional subset consisting of points whose orbit length is not maximal. As in the case of the usual action of the symmetric group, \ these points lie on linear subspaces. 

Our main result is Theorem~\ref{thm:Main} which bounds the dimension of subspaces one has to consider to decide  non-negativity of a multisymmetric  polynomial.
Besides this general bound the idea of the proof can be adjusted in particular situations to derive stronger bounds. We give several other bounds to illustrate this. 

As an application of our results we investigate the question of deciding if a given symmetric or multisymmetric polynomial defines a convex function.
It is straightforward to observe that the question of convexity of a $k$-symmetric polynomial in $kn$ many variables leads to the question of certifying whether a $2k$-symmetric polynomial in $2kn$ variables is non-negative. Consequently,   we show in Theorems~\ref{thm:Hf} and~\ref{thm:convex2} that our results on non-negativity of multisymmetric polynomials imply in particular that for (multi-)symmetric polynomials of a fixed degree the complexity of deciding convexity does depend polynomially on  the number of variables.

The article is structured as follows. In the next section we will give a brief introduction to the theory of multisymmetric polynomials and provide a relation between $k$-symmetric polynomials and $k$-variate polynomials which will play a crucial role in our arguments. With these preliminaries at hand we are able to state and prove the main result in Section 3. The last section is then devoted to the application of the non-negativity result to the problem of deciding convexity. This section  also includes some refinements of our main results which apply to this setting.

\section{Multisymmetric Polynomials}
For $n\in\N$ let $S_n$ denote the symmetric group on $n$ elements which acts on the $n$-dimensional vector space $V:=\R^n$ by permuting coordinates. For every $k\in\N$ we can consider the diagonal action of $S_n$ on the vector space $$V^{k}:=\bigoplus_{i=1}^{k}V$$ of $k$-tuples of vectors from $V$. 
The action of $S_n$ extends to the ring $\rr$, which is just the polynomial ring $\R[X_{11},\dots,X_{nk}]$ after identifying $X_{i1},\dots,X_{ik}$ with the standard basis of the $i$-th direct summand of $V^k$. It is convenient to think of these variables as an $n\times k$ array as indicated in the following notation.

\begin{note}
Throughout this paper let
                        \[X:=\begin{pmatrix} X_{11}&X_{12}&\ldots& X_{1k}\\
                        X_{21}&X_{22}&\ldots& X_{2k}\\
                        \vdots&\vdots&\cdots&\vdots\\
                        X_{n1}&X_{n2}&\ldots& X_{nk}\\
                        \end{pmatrix}\]
denote an $n\times k$ array of variables. By $X_{i\cdot}$ and $X_{\cdot j}$ we denote the i-th row of $X$ and the j-th column of $X$, respectively.
\end{note}
Following this notation, $S_n$ permutes the rows of $X$.
The invariant ring of the polynomial ring $\rr = \R[X_{11},\dots,X_{nk}]$ with respect to this action is the algebra of $k$-symmetric polynomials, denoted by $\rx$.
Alternatively, $\rx$ can be thought of as the $n$-fold symmetric product of the polynomial $\R$-algebra in $k$ variables, which is a classically studied object (see \cite{sch,junker_93}).
Note that for $k=1$ 
this is
the algebra of symmetric polynomials, which is a polynomial ring, i.e., there are $n$ algebraically independent polynomials $p_1,\ldots,p_n$ such that $\R[X_1,\ldots,X_n]^{S_n}=\R[p_1,\ldots,p_n]$. In the case $k>1$, $n>1$ the symmetric group $S_n$ is not operating as a finite reflection group. Therefore, it follows from the  classical Chevalley-Shephard-Todd Theorem (see for example \cite{Chev}) that whenever $k>1$ and $n>1$  the algebra of $k$-symmetric polynomials in $nk$ variables is no longer a polynomial ring. However, for all $k$ it is finitely generated as $\R$-algebra \cite{Bri,dalbec}. For our purposes the representation in terms of so called multisymmetric power sums will be crucial.

\begin{definition}
Given a polynomial $f\in\rr$, we denote its symmetrization, that is, the sum over its $S_n$-orbit, by $\sym(f)\in\rx$. For all $\alpha\in\N^k$ we define the (multisymmetric) power sum
  \[
  p_\alpha:=\sym(X_{1\cdot}^\alpha),
  \]
  where $X_{1\cdot}^\alpha:=X_{11}^{\alpha_1}\cdots X_{1k}^{\alpha_k}$.
\end{definition}

This is is a generalization of the power sum polynomials used in \cite{Rie2}, where the symmetric case $k=1$ is considered. In that case, the first $n$ power sum polynomials form an algebraically independent set generating the $\R$-algebra of symmetric polynomials. In the $k$-symmetric case, this  statement does not generalize, i.e., the generators are no longer algebraically independent.

In the following, we study the $k$-symmetric power sums and their connection to $k$-symmetric functions. 

\begin{definition}\label{def:wDeg}
Let $w:=(w_1, \dots, w_k)$ be a $k$-tuple of positive integers. We consider the grading on the $\R$-algebra $\R[Y_1,\dots,Y_k]$ given by defining $Y_j$ to be homogeneous of degree $w_j$. This grading also induces a grading on the $\R$-algebra $\rr$ by the algebra-homomorphism $\varphi:\rr \to \R[Y_1,\dots,Y_k]$, $X_{ij} \mapsto Y_j$. Alternatively, the latter grading is given by defining each $X_{ij}$ to be homogeneous of degree $w_j$.

The degree $\deg_w (f)$ of an element $f \in \R[Y_1,\dots,Y_k]$ (resp. $f \in \rr$) with respect to the above grading  is called the $w$-weighted degree, or simply $w$-degree, of $f$.
\end{definition}

Alternatively, one can define the $w$-degree on the monomials of $\R[Y_1,\ldots,Y_k]$ and $\rr$ by $\deg_w(Y^{\alpha}):=\sum_{j=1}^{k}w_j\alpha_{j}$ and 
{$\deg_w(X_{1\cdot}^{\alpha^{(1)}}\cdots X_{n\cdot}^{\alpha^{(n)}}):=\sum_{i=1}^{n}\sum_{j=1}^{k}w_j\alpha_j^{(i)}$}

, respectively, where {$\alpha,\alpha^{(1)},\ldots,\alpha^{(n)}\in\N^k$}. 
After that, one extends this definition to a polynomial $f$ by taking the maximal $w$-degree of all monomials of $f$. Note that we retrieve the usual degree from this definition by setting all weights $w_j$ equal to~$1$.
\begin{ex}\label{ex:1}
Fix $n\in\N$ rather large and consider for all (nonzero) parameters $\gamma\in\R^7 $ the $2$-symmetric polynomial given by 
\begin{multline*}
f(X):=
\gamma_1\sum_iX_{i1}^4
+\gamma_2\left(\sum_i\sum_jX_{i1}X_{j1}\right)^2
-\gamma_3\left(\sum_i\sum_jX_{i1}X_{j1}\right)\left(\sum_iX_{i1}X_{i2}\right)\\
-\gamma_4\left(\sum_iX_{i1}^3\right)\left(\sum_iX_{i2}\right)
-\gamma_5\left(\sum_iX_{i1}\right)\left(\sum_iX_{i2}\right)^2
+\gamma_6\left(\sum_iX_{i1}\right)^2
+\gamma_7\left(\sum_iX_{i1}X_{i2}\right),
\end{multline*}
where all sums go from $1$ to $n$. We have $\deg_{(1,1)}(f)=\deg(f)=4$. Recognizing that the exponents of the column $X_{\cdot 2}$ are small compared to the exponents of $X_{\cdot 1}$ one could give more weight to the second column, e.\,g. by considering the $(3,5)$-degree: $\deg_{(3,5)}(f)=14$.
\end{ex}
\begin{note}\label{not:morphismus,Mf,Ef}
Fix some weights $w$. For a polynomial $f \in \rr$ we define $M_f\subset \rr$ to be the set of monomials of $f$. By construction, $\varphi: \rr \to \R[Y_1, \dots, Y_k]$, $X_{ij} \mapsto Y_j$ is a morphism of graded $\R$-algebras, considering the gradings as above. Thus, 
\[
\deg_w (f) = \max\{\deg_w (m) \mid m \in \varphi(M_f)\} = \max\{w^T\alpha \mid \alpha \in E_f\},
\]
where $E_f\subset \N^k$ is the set of exponent tuples of all monomials in $\varphi(M_f)$. Note however, that in general $\deg_w (f) \neq \deg_w (\varphi(f))$, {as $\varphi(M_f)$ may differ from $M_{\varphi(f)}$}.
\end{note}
 Let $\deg_w(f)=d$. Then the interpretation above gives rise to a very useful view on the $w$-degree as a hyperplane defining a simplex $\{y\in\R_{\geq 0}^k\mid w^Ty\leq d\}$ enclosing $E_f$. Applied to Example~\ref{ex:1}, where $k=2$, we can consider Figure~\ref{fig:1}. Note that $\varphi(M_f)=\{Y_1^4,Y_1^3Y_2,Y_1^2,Y_1Y_2^2,Y_1Y_2\}$ and hence, $E_f=\{(4,0),(3,1),(2,0),(1,2),(1,1)\}$. Both choices of the weights $w$ in Example~\ref{ex:1} define a triangle enclosing $E_f$.
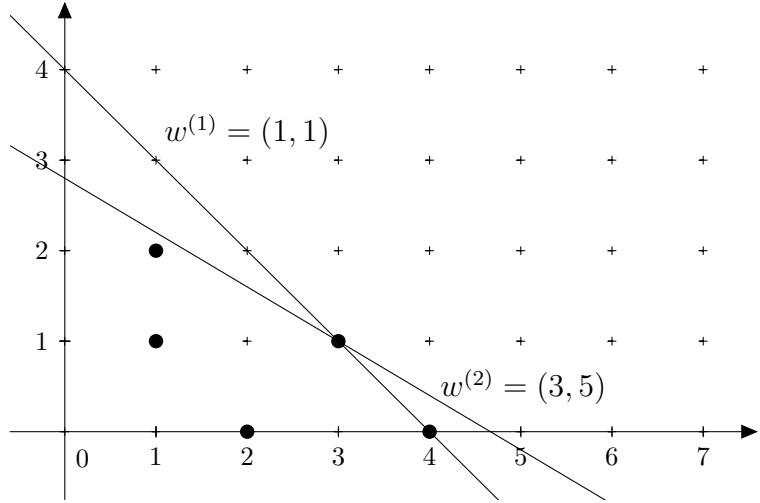
\begin{figure}
\begin{tikzpicture}[line cap=round,line join=round,>=triangle 45,x=1.2cm,y=1.2cm]
\draw[->,color=black] (-0.6,0.0) -- (7.6,0.0);
\foreach \x in {,1,2,3,4,5,6,7}
\draw[shift={(\x,0)},color=black] (0pt,2pt) -- (0pt,-2pt) node[below] {\footnotesize $\x$};
\draw[->,color=black] (0.0,-0.75) -- (0.0,4.75);
\foreach \y in {,1,2,3,4}
\draw[shift={(0,\y)},color=black] (2pt,0pt) -- (-2pt,0pt) node[left] {\footnotesize $\y$};
\draw[color=black] (0pt,-10pt) node[right] {\footnotesize $0$};
\clip(-0.6,-0.75) rectangle (7.6,4.75);
\draw [domain=-0.6:7.6] plot(\x,{-1*\x+4});
\draw [color=black] (2,3) node[above] {{$w^{(1)}=(1,1)$}};
\draw [domain=-0.6:7.6] plot(\x,{-0.6*\x+2.8});
\draw [color=black] (4,0.5) node[right] {{$w^{(2)}=(3,5)$}};

\draw [color=black] (0.0,0.0)-- ++(-1.5pt,0 pt) -- ++(3.0pt,0 pt) ++(-1.5pt,-1.5pt) -- ++(0 pt,3.0pt);
\draw [color=black] (1.0,0.0)-- ++(-1.5pt,0 pt) -- ++(3.0pt,0 pt) ++(-1.5pt,-1.5pt) -- ++(0 pt,3.0pt);
\draw [fill=black] (2.0,0.0) circle (2.5pt);
\draw [color=black] (3.0,0.0)-- ++(-1.5pt,0 pt) -- ++(3.0pt,0 pt) ++(-1.5pt,-1.5pt) -- ++(0 pt,3.0pt);
\draw [fill=black] (4.0,0.0) circle (2.5pt);
\draw [color=black] (5.0,0.0)-- ++(-1.5pt,0 pt) -- ++(3.0pt,0 pt) ++(-1.5pt,-1.5pt) -- ++(0 pt,3.0pt);
\draw [color=black] (6.0,0.0)-- ++(-1.5pt,0 pt) -- ++(3.0pt,0 pt) ++(-1.5pt,-1.5pt) -- ++(0 pt,3.0pt);
\draw [color=black] (7.0,0.0)-- ++(-1.5pt,0 pt) -- ++(3.0pt,0 pt) ++(-1.5pt,-1.5pt) -- ++(0 pt,3.0pt);
\draw [color=black] (0.0,1.0)-- ++(-1.5pt,0 pt) -- ++(3.0pt,0 pt) ++(-1.5pt,-1.5pt) -- ++(0 pt,3.0pt);
\draw [fill=black] (1.0,1.0) circle (2.5pt);
\draw [color=black] (2.0,1.0)-- ++(-1.5pt,0 pt) -- ++(3.0pt,0 pt) ++(-1.5pt,-1.5pt) -- ++(0 pt,3.0pt);
\draw [fill=black] (3.0,1.0) circle (2.5pt);
\draw [color=black] (4.0,1.0)-- ++(-1.5pt,0 pt) -- ++(3.0pt,0 pt) ++(-1.5pt,-1.5pt) -- ++(0 pt,3.0pt);
\draw [color=black] (5.0,1.0)-- ++(-1.5pt,0 pt) -- ++(3.0pt,0 pt) ++(-1.5pt,-1.5pt) -- ++(0 pt,3.0pt);
\draw [color=black] (6.0,1.0)-- ++(-1.5pt,0 pt) -- ++(3.0pt,0 pt) ++(-1.5pt,-1.5pt) -- ++(0 pt,3.0pt);
\draw [color=black] (7.0,1.0)-- ++(-1.5pt,0 pt) -- ++(3.0pt,0 pt) ++(-1.5pt,-1.5pt) -- ++(0 pt,3.0pt);
\draw [color=black] (0.0,2.0)-- ++(-1.5pt,0 pt) -- ++(3.0pt,0 pt) ++(-1.5pt,-1.5pt) -- ++(0 pt,3.0pt);
\draw [fill=black] (1.0,2.0) circle (2.5pt);
\draw [color=black] (2.0,2.0)-- ++(-1.5pt,0 pt) -- ++(3.0pt,0 pt) ++(-1.5pt,-1.5pt) -- ++(0 pt,3.0pt);
\draw [color=black] (3.0,2.0)-- ++(-1.5pt,0 pt) -- ++(3.0pt,0 pt) ++(-1.5pt,-1.5pt) -- ++(0 pt,3.0pt);
\draw [color=black] (4.0,2.0)-- ++(-1.5pt,0 pt) -- ++(3.0pt,0 pt) ++(-1.5pt,-1.5pt) -- ++(0 pt,3.0pt);
\draw [color=black] (5.0,2.0)-- ++(-1.5pt,0 pt) -- ++(3.0pt,0 pt) ++(-1.5pt,-1.5pt) -- ++(0 pt,3.0pt);
\draw [color=black] (6.0,2.0)-- ++(-1.5pt,0 pt) -- ++(3.0pt,0 pt) ++(-1.5pt,-1.5pt) -- ++(0 pt,3.0pt);
\draw [color=black] (7.0,2.0)-- ++(-1.5pt,0 pt) -- ++(3.0pt,0 pt) ++(-1.5pt,-1.5pt) -- ++(0 pt,3.0pt);
\draw [color=black] (0.0,3.0)-- ++(-1.5pt,0 pt) -- ++(3.0pt,0 pt) ++(-1.5pt,-1.5pt) -- ++(0 pt,3.0pt);
\draw [color=black] (1.0,3.0)-- ++(-1.5pt,0 pt) -- ++(3.0pt,0 pt) ++(-1.5pt,-1.5pt) -- ++(0 pt,3.0pt);
\draw [color=black] (2.0,3.0)-- ++(-1.5pt,0 pt) -- ++(3.0pt,0 pt) ++(-1.5pt,-1.5pt) -- ++(0 pt,3.0pt);
\draw [color=black] (3.0,3.0)-- ++(-1.5pt,0 pt) -- ++(3.0pt,0 pt) ++(-1.5pt,-1.5pt) -- ++(0 pt,3.0pt);
\draw [color=black] (4.0,3.0)-- ++(-1.5pt,0 pt) -- ++(3.0pt,0 pt) ++(-1.5pt,-1.5pt) -- ++(0 pt,3.0pt);
\draw [color=black] (5.0,3.0)-- ++(-1.5pt,0 pt) -- ++(3.0pt,0 pt) ++(-1.5pt,-1.5pt) -- ++(0 pt,3.0pt);
\draw [color=black] (6.0,3.0)-- ++(-1.5pt,0 pt) -- ++(3.0pt,0 pt) ++(-1.5pt,-1.5pt) -- ++(0 pt,3.0pt);
\draw [color=black] (7.0,3.0)-- ++(-1.5pt,0 pt) -- ++(3.0pt,0 pt) ++(-1.5pt,-1.5pt) -- ++(0 pt,3.0pt);
\draw [color=black] (0.0,4.0)-- ++(-1.5pt,0 pt) -- ++(3.0pt,0 pt) ++(-1.5pt,-1.5pt) -- ++(0 pt,3.0pt);
\draw [color=black] (1.0,4.0)-- ++(-1.5pt,0 pt) -- ++(3.0pt,0 pt) ++(-1.5pt,-1.5pt) -- ++(0 pt,3.0pt);
\draw [color=black] (2.0,4.0)-- ++(-1.5pt,0 pt) -- ++(3.0pt,0 pt) ++(-1.5pt,-1.5pt) -- ++(0 pt,3.0pt);
\draw [color=black] (3.0,4.0)-- ++(-1.5pt,0 pt) -- ++(3.0pt,0 pt) ++(-1.5pt,-1.5pt) -- ++(0 pt,3.0pt);
\draw [color=black] (4.0,4.0)-- ++(-1.5pt,0 pt) -- ++(3.0pt,0 pt) ++(-1.5pt,-1.5pt) -- ++(0 pt,3.0pt);
\draw [color=black] (5.0,4.0)-- ++(-1.5pt,0 pt) -- ++(3.0pt,0 pt) ++(-1.5pt,-1.5pt) -- ++(0 pt,3.0pt);
\draw [color=black] (6.0,4.0)-- ++(-1.5pt,0 pt) -- ++(3.0pt,0 pt) ++(-1.5pt,-1.5pt) -- ++(0 pt,3.0pt);
\draw [color=black] (7.0,4.0)-- ++(-1.5pt,0 pt) -- ++(3.0pt,0 pt) ++(-1.5pt,-1.5pt) -- ++(0 pt,3.0pt);
\end{tikzpicture}
\caption{Visualization of $w$-weighted degree}
\label{fig:1}
\end{figure}

In the following Theorem, we see that the lattice points contained in the simplex which is given by some weights $w$ have an emerging meaning when we consider the according power sums. 
\begin{thm}\label{thm: generated by p alpha}
Let $w = (w_1,\ldots,w_k)$ be a $k$-tuple of positive integers and let $d\in \N$. The $\R$-algebra generated by all power sums $p_\alpha$ with $|\alpha|_w:=w^T\alpha \leq d$ contains  all $k$-symmetric polynomials of $w$-degree at most $d$.
\end{thm}

\begin{proof}
Reviewing the proof of \cite[Thm.\,1.2]{dalbec} yields the assertion:

It is enough to show that the first power sums generate all $k$-symmetric monomial functions
\[
{m_{\alpha^{(1)},\ldots,\alpha^{(\ell)}}:=\sym(X_{1\cdot}^{\alpha^{(1)}}\cdots X_{\ell\cdot}^{\alpha^{(\ell)}})}
\]
of $w$-degree at most $d$, where $\ell\in\{1,\ldots,n\}$. For {$\alpha^{(0)},\ldots,\alpha^{(\ell)}\in\N^k$} 
and some {positive} integer $c$ the following equality holds:
\[
{c\cdot m_{\alpha^{(0)},\ldots,\alpha^{(\ell)}}=p_{\alpha^{(0)}}m_{\alpha^{(1)},\ldots,\alpha^{(\ell)}}-\sum\limits_{i=1}^{\ell}m_{\alpha^{(1)},\ldots,\alpha^{(i)}+\alpha^{(0)},\ldots,\alpha^{(\ell)}}.}
\]
Since {$\deg_w(m_{\alpha^{(0)},\ldots,\alpha^{(\ell)}})=\sum_{j=0}^{\ell} |\alpha^{(j)}|_w\leq d$}
, there are only polynomials of $w$-degree equal or less than $d$ on the right-hand side of this identity. Hence, we can write a monomial function with $\ell+1$ exponent tuples as combination of a power sum of $w$-degree at most $d$ and some monomial functions with $\ell$ exponent tuples, both of $w$-degree at most $d$. Noticing that a monomial function with only one exponent tuple is a power sum ends the proof.
\end{proof}

Reconsider Figure~\ref{fig:1}. Taking the lattice points contained in the simplices given by $w_1$ and $w_2$, respectively, we get two sets of power sums, each of them sufficient to describe $f$ as a polynomial expression in its elements. In fact, the proof of Theorem~\ref{thm: generated by p alpha} shows even more.
\begin{remark}
Let $f\in\rx$. For all $k$-tuples $w$ let $d_w:=\deg_w(f)$. Then $f$ can be written as a polynomial expression in the power sums associated to the lattice points that are contained in the intersection $\bigcap\limits_w \{y\in\R_{\geq 0}^k\mid w^Ty\leq d_w\}.$ In particular, for  fixed $d_w$,  one can deduce that the number of coefficients in a representation of such a polynomial will  (for large values of $n$ ) be bounded by a constant. 
\end{remark}
The reason why we are interested in a different way of describing $k$-symmetric polynomials is illuminated by the following observation. For all $\alpha\in\N^k$ the {partial derivative with respect to $X_{ij}$}
of $p_\alpha$ is a polynomial in the variables $X_{i\cdot}$. Moreover, for every $j\in\{1,\ldots,k\}$ there exists one polynomial $r_j\in\R[Y_1,\ldots,Y_k]$ such that
\[
\partial_{ij}p_\alpha=r_j(X_{i\cdot}).
\]
Note that $r_j$ does not depend on $i$. This makes the derivatives of power sums easier to handle than the derivatives of the usual generaing set, that is the set of monomials of $w$-degree at most $d$. 
{We immediately conclude the following result for linear combinations of power sums:
\begin{prop} \label{prop:derivPAlpha}
	Let $w = (w_1,\dots,w_k)$ be a $k$-tuple of positive integers. Denote $N_d:=\{\alpha\in\N^k\mid |\alpha|_w\leq d\}$ and let $u\in\R^{N_d}$. Then there are $k$ polynomials $\tilde{q}_1,\ldots,\tilde{q}_k\in\R[Y_1,\ldots,Y_k]$ such that for all $1 \leq i \leq n$, $1 \leq j \leq k$ we have
\[\partial_{ij}\left(\sum_{\alpha\in N_d}u_\alpha p_\alpha\right)=\tilde{q}_j(X_{i\cdot})\]
and $\tilde q_j$ is of $w$-degree at most $d-w_j$ for each $j\in\{1,\ldots,k\}$.
\end{prop} }
We will use this fact in the proof of Theorem~\ref{thm:Main}.

\section{Positivity of multisymmetric polynomials}
\begin{definition}
    For a positive integer $m$ let $A_m$ denote the subset of $\R^{n \times k}$ consisting of all points $x=\left(x_{ij}\right)$ with at most $m$ distinct rows:
    \[A_m := \left\{x \in \R^{n \times k}\mid\#\left\{x_{1\cdot}, \ldots, x_{n\cdot}\right\} \leq m\right\}.\]

    For { $f\in \rr$} {or, more generally, $f \in C^0(V^k)$} 
we define $\kappa(f)$ to be the smallest positive integer such that
    \[\min_{x \in B_r}f(x)=\min_{x\in B_r\cap A_{\kappa(f)}}f(x)\]
    holds for all $r \geq 0$, where $B_r := \{x\in\R^{n\times k}\mid \sum_{i=1}^n\sum_{j=1}^k x_{ij}^2=r\}$.
\end{definition}Note that we always have $\kappa(f)\leq n$ and that by the definition above we see that low values of  $\kappa(f)$ imply that non-negativity of $f$ can be checked on sets of small dimension.  
In particular the above generalizes the previous setup of symmetric polynomials and in this case it is known that $\kappa(f) \leq \max\{2,\left\lfloor\frac{\deg f}{2}\right\rfloor\}$  (see \cite{timofte-2003,Rie}). In Theorem~\ref{thm:Main} below  we will show a bound for a $k$-symmetric polynomial $f\in \rx$ in terms of the (weighted) degree of $f$.
{For technical reasons in the proofs to come we observe that $\kappa$ is lower semi-continuous:

\begin{prop} \label{prop:semiCont}
	Let $(f_{\ell})_{\ell\in\N} \subset C^0(V^k)$ be a sequence of continuous functions converging uniformly on compact sets to some $f \in C^0(V^k)$. Then
    \[\kappa(f) \leq \liminf_{\ell\to\infty} \kappa(f_\ell). \]
\end{prop}

\begin{proof}
	The sequence $\restr{f_\ell}{B_r}$ 
	converges uniformly to $\restr{f}{B_r}$.
	 This implies that for all $m\in \N$, $\min_{x \in B_r}f_\ell(x)$ and $\min_{x \in B_r \cap A_m} f_\ell(x)$ converge to $\min_{x \in B_r}f(x)$ and $\min_{x \in B_r \cap A_m}f(x)$, respectively.
\end{proof}
}
By definition of $\kappa(f)$ we immediately see:
{\begin{prop}\label{prop:prop}
Let {$f \in \rr$}.
\begin{enumerate}[(i)]
   \item If $f\geq 0$ on $A_{\kappa(f)}$, then $f\geq 0$ on $\R^{n\times k}$.
   \item If $f>0$ on $A_{\kappa(f)}$, then $f>0$ on $\R^{n\times k}$.
   \item If $f\neq 0$ on $A_{\kappa(f)}$, then $f\neq 0$ on $\R^{n\times k}$.
\end{enumerate}
\end{prop}}

\begin{definition}
Let $n\in\N$. A tuple $\lambda:=(\lambda_1,\ldots,\lambda_\ell)$ of $\ell$ positive integers such that $\lambda_1\geq\lambda_2\geq\ldots\geq \lambda_\ell$ and 
$n=\lambda_1+\ldots+\lambda_\ell$ is called an $\ell$-partition of $n$. We will write $\lambda\vdash_{\ell} n$ to say that $\lambda$ is such a partition.
\end{definition}
The following Proposition gives a rough estimate on the number of $\ell$-partitions for a fixed natural number $n$.
\begin{prop} \label{prop:numberPart}
Let $n\in\N$. Then for every $\ell\in\N$ the number of $\ell$-partitions of $n$ is bounded by $n^\ell$.
\end{prop}
Algorithmically checking for the global properties in Proposition~\ref{prop:prop} is an instance of a decision problem. As mentioned, it is known that deciding global positivity of a multivariate polynomial is NP-hard in general (see for example \cite{murty}). {Our aim is to exploit the structure of $k$-symmetric polynomials $f$ by bounding $\kappa(f)$.}

Observe that $A_{m}$ is a union of $km$-dimensional subspaces, each of which corresponds to one particular way of assigning the $m$ distinct rows. It follows that modulo the action of $S_n$ on the rows each of these choices is uniquely represented by an $m$-partition of $n$. Therefore, the statements in Proposition~\ref{prop:prop} amount to saying that all of the mentioned global properties can be checked by verifying them on each of the {$k\cdot\kappa(f)$-dimensional} subspaces corresponding to the various $\kappa(f)$-partitions of $n$. The number of such partitions can be bounded by $n^{\kappa(f)}$, so if $\kappa(f)$ can be bounded by a quantity independent of $n$ for a certain family of polynomials $f$, this implies that the complexity of testing for those global properties {for $k$-symmetric polynomials} grows only polynomially in the number of variables -- in  contrast to the general case.

Our main result is now presented in the following Theorem.

\begin{thm}\label{thm:Main}
    Let $w=(w_1,\dots, w_k)$ be a $k$-tuple of positive integers. Let $d \geq \max\{2w_j \mid 1\leq j \leq k\}$ and let $f\in\rx$ be a $k$-symmetric polynomial of $w$-weighted degree at most $d$. Then
    \[\kappa(f) \leq \prod_{j=1}^k \left\lfloor \frac{d}{w_j}\right\rfloor.\]
\end{thm}
Before we give the proof of Theorem \ref{thm:Main} in detail, we shortly outline the main idea. The proof relies essentially on the  classical Lagrange multiplier rule. For every $r>0$ we minimize the given multisymmetric function on a sphere of radius $r$. Since every  sphere is smooth we can infer the existence of Lagrange multipliers. Using a perturbation we can guarantee that the system of polynomial equations, which certify this existence, describe a zero-dimensional variety. Finally, the bound on $\kappa(f)$ given in the Theorem above follows by analysing the finitely many solutions. 

\begin{proof}
Set $d_j := \left\lfloor d/w_j\right\rfloor$ and $\mu := \prod_{j=1}^k d_j$. Define $g:=\sum_{i=1}^n\sum_{j=1}^k X_{ij}^{d_j+1}$. By Proposition~\ref{prop:semiCont} it suffices to show that $\kappa(f_\varepsilon) \leq \mu$ for all $\varepsilon > 0$, where $f_\varepsilon := f + \varepsilon g$.

Fix $r > 0$ and let $\varepsilon > 0$. Consider a point $z^* \in B_r$ where $\restr{f_\varepsilon}{B_r}$ is minimized. We have to show $z^* \in A_\mu$.

Henceforth, we denote $N_d:=\{\alpha\in\N^k\mid |\alpha|_w\leq d\}$ and $I_{n,k}:=\{1,\ldots,n\}\times\{1,\ldots,k\}$. Define $p:=\sum_{i=1}^n\sum_{j=1}^k X_{ij}^{2}$ and denote $$\nabla:\rr\to\rr^{I_{n,k}}, h\mapsto \left(\partial_{11}h,\ldots,\partial_{nk}h \right).$$
Since $z^*$ is a minimum point of $f_\varepsilon$ on {$B_r=\{x\in \R^{n\times k} \mid p(x) = r\}$} and $\nabla p(z^*)$ is not zero, there exists a Lagrange multiplier $\lambda\in\R$ such that
\begin{equation}\label{eq: lagrange1}
	\nabla f_\varepsilon(z^*)+\lambda\nabla p(z^*)=0.
\end{equation}
Since $f,p\in\rx$, by Theorem~\ref{thm: generated by p alpha} there are polynomials $F,P\in\R[(Z_\alpha)_{\alpha\in N_d}]$ such that $f=F((p_\alpha)_{\alpha\in N_d})$ and $p=P((p_\alpha)_{\alpha\in N_d})$. We define the matrix polynomial  $M\in\rr^{N_d\times I_{n,k}}$ by
\[
M_{\alpha,(i,j)}:=\partial_{ij}p_\alpha.
\]
Let $c\in\R^{N_d}$ be such that $c_\alpha:=p_\alpha(z^*)$. Then by the chain.rule we can write \eqref{eq: lagrange1} as
\[
\varepsilon\nabla g(z^*)+\left(\nabla F(c) + \lambda\nabla P(c)\right)M(z^*)=0.
\]
In other words, if we set $u := \left(\nabla F(c) + \lambda\nabla P(c)\right)^T \in\R^{N_d}$, then $z^*$ is a solution to the system of polynomial equations
\begin{equation} \label{eq: lagrange2}
\varepsilon\nabla g(X)+u^TM(X)=0.
\end{equation}

By Proposition~\ref{prop:derivPAlpha} there are polynomials $\tilde q_1, \dots, \tilde q_k \in \R[Y_1,\dots, Y_k]$ such that the \mbox{$(i,j)$-th} entry of $u^TM(X)$ is $\tilde q_j(X_{i\cdot})$ and $\deg_w(\tilde q_j) \leq d-w_j$. Thus, \eqref{eq: lagrange2} can be rewritten as $q_j(X_{i\cdot}) = 0$ for all $(i,j)\in I_{n,k}$, where 
\[
q_j:=\varepsilon(d_j+1)Y_j^{d_j}+\tilde{q_j}\in\R[Y_1,\ldots,Y_k]\quad\mbox{for all}\;j\in\{1,\ldots,k\}.
\]
In other words, each row $z_{1\cdot}^*,\dots,z_{n\cdot}^*$ of $z^*$ is a point of the complex zero set \[\mathcal{V}(q_1,\ldots,q_k) = \{y\in \C^k \mid q_1(y) =0,\ldots, q_k(y)= 0\} \subset \C^k.\]

However, the coordinate ring {$\C[Y_1,\ldots,Y_k]/(q_1,\ldots,q_k)$} is a $\C$-vector space of dimension at most $\prod_{j=1}^k d_j = \mu$. Indeed, because of {$\deg_w(Y_j^{d_j}) > \deg_w(\tilde q_j)$}
{ all monomials $Y^\alpha$ with $\alpha_j > d_j$ for some $1\leq j\leq k$ can be rewritten as a sum of monomials of smaller $w$-degree.}
Hence, $\#\mathcal{V}(q_1,\ldots,q_k)\leq \mu$, so at most $\mu$ of the rows $z_{1\cdot}^*,\ldots,z_{n\cdot}^*$ can be different, that is, $z^*\in A_\mu$.
\end{proof}

{
\begin{remark}
	Note that the proof of Theorem \ref{thm:Main} given above works in exactly the same manner for a more general setting: If $f\in C^1(V^k)$ and there exists $F \in C^1(\R^{N_d})$ such that $f = F((p_\alpha)_{\alpha\in N_d})$, then we get the same bound for $\kappa(f)$. Additionally, we can further extend this bound by Proposition~\ref{prop:semiCont} to functions of the form $f = F((p_\alpha)_{\alpha\in N_d})$ with $F \in C^0(\R^{N_d})$. 
\end{remark}

\begin{ex}
	We consider the case $k=2$. The family of polynomials $$f_1^{m,\ell} := p_{(2,1)}^m p_{(1,1)}^2 - p_{(0,3)}^\ell \quad (m,\ell \in \N)$$ is not bounded with respect to any weighted degree. However, all $f_1^{m,\ell}$ can be written in terms of the power sum polynomials of (usual) degree at most 3. Thus, by the preceding remark, we still get the bound $\kappa(f_1^{m, \ell}) \leq 9$ for all $m, \ell \in \N$. Even for the rational function $f_2 := (p_{(3,0)}-2p_{(1,2)}p_{(2,1)})/(p_{(2,0)}^2 + 1)$ and the functions $f_3^m(x) := |\exp(f_2(x))-2|+f_1^{m,m}(x)$ ($m \in \N$) we get that the $\kappa$-value is at most 9.
\end{ex}
}

\begin{remark}
Note that in the proof of Theorem~\ref{thm:Main} the bound for $\kappa(f)$ is given by an upper bound on the number of complex solutions of a certain system of polynomial equations. Since in fact we are only interested in the number of real solutions, {$\mu$ might be chosen significantly smaller depending on the specific representation $F$. If, for instance, $f$ is a polynomial with only few monomials but of high degree, 
}
it could be advantageous  to argue with  Khovanskii's fewnomial bound. This result bounds  the number of isolated real solutions of a polynomial system by a function of the number of distinct monomials which are involved rather than by the a function of the degree (see e.g., \cite[Theorem 3.7]{sturmfels}).
\end{remark}

For a very rough estimation of $\kappa(f)$ we can just use the usual degree and get the following upper bound without any effort.

\begin{cor}\label{cor:dk}
	If $f\in\rx$ is a $k$-symmetric polynomial of degree $d\geq 2$, then 
    \[\kappa(f) \leq d^k.\]
\end{cor}

However, Theorem~\ref{thm:Main} is much stronger. Recall that, using Notation~\ref{not:morphismus,Mf,Ef}, the condition that $f$ is of $w$-degree at most $d$ can be expressed as $E_f \subset \{y\in\R_{\geq 0}^k\mid w^Ty\leq d\}$ and note that the hyperplane $\{y\in\R^k\mid w^Ty= d\}$ defining this enclosing simplex intersects the $j$-th coordinate axis at the point {$(d/w_j) e_j$} (where $e_1,\dots, e_k$ is the standard basis of $\R^k$). This gives the following geometric reformulation of the Theorem ~\ref{thm:Main}.

\begin{thm} \label{thm:reformulation}
	Let $f\in\rx$ be a $k$-symmetric polynomial and let $a_1,\ldots,a_k \in \Q_{\geq 2}$ such that $E_f$ is contained in the simplex $$\Delta(a_1,\dots,a_k) := {\rm conv}(0,a_1e_1,\dots,a_ke_k) {\subset \R^k}.$$ Then $\kappa(f)$ is bounded by {$\prod_{j=1}^k \lfloor a_j \rfloor$, i.e.} the number of lattice points in $[0,a_1-1]\times\dots\times[0,a_k-1]$.
\end{thm}

Note that this number of lattice points is approximately $k!\,{\rm vol}(\Delta(a_1,\dots,a_k))$. Therefore, finding a good bound for $\kappa(f)$ via Theorem~\ref{thm:reformulation} roughly amounts to finding the smallest simplex $\Delta(a_1,\dots,a_k)$ enclosing $E_f$.

\begin{ex}[Example~\ref{ex:1} continued]
Reconsider Figure~\ref{fig:1}. The triangle described by {$w^{(2)}=(3,5)$} minimizes the number of lattice points. Counting the lattice points in the rectangle drawn in Figure~\ref{fig:2} gives the upper bound $\kappa(f)\leq 8.$ This bound holds for all $n$ and for all choices of the parameters in Example~\ref{ex:1}. 
\begin{figure}
\begin{tikzpicture}[line cap=round,line join=round,>=triangle 45,x=1.2cm,y=1.2cm]
\draw[->,color=black] (-0.6,0.0) -- (7.6,0.0);
\foreach \x in {,1,2,3,4,5,6,7}
\draw[shift={(\x,0)},color=black] (0pt,2pt) -- (0pt,-2pt) node[below] {\footnotesize $\x$};
\draw[->,color=black] (0.0,-0.75) -- (0.0,3.75);
\foreach \y in {,1,2,3}
\draw[shift={(0,\y)},color=black] (2pt,0pt) -- (-2pt,0pt) node[left] {\footnotesize $\y$};
\draw[color=black] (0pt,-10pt) node[right] {\footnotesize $0$};
\clip(-0.6,-0.75) rectangle (7.6,3.75);

\draw [domain=-0.6:7.6] plot(\x,{-0.6*\x+2.8});

\fill[line width=0.0pt,fill=black,fill opacity=0.1] (-0.0,1.8) -- (3.6666,1.8) -- (3.6666,-0.0) -- (-0.0,-0.0) -- cycle;

\draw [color=black] (0.0,0.0)-- ++(-1.5pt,0 pt) -- ++(3.0pt,0 pt) ++(-1.5pt,-1.5pt) -- ++(0 pt,3.0pt);
\draw [color=black] (1.0,0.0)-- ++(-1.5pt,0 pt) -- ++(3.0pt,0 pt) ++(-1.5pt,-1.5pt) -- ++(0 pt,3.0pt);
\draw [fill=black] (2.0,0.0) circle (2.5pt);
\draw [color=black] (3.0,0.0)-- ++(-1.5pt,0 pt) -- ++(3.0pt,0 pt) ++(-1.5pt,-1.5pt) -- ++(0 pt,3.0pt);
\draw [fill=black] (4.0,0.0) circle (2.5pt);
\draw [color=black] (5.0,0.0)-- ++(-1.5pt,0 pt) -- ++(3.0pt,0 pt) ++(-1.5pt,-1.5pt) -- ++(0 pt,3.0pt);
\draw [color=black] (6.0,0.0)-- ++(-1.5pt,0 pt) -- ++(3.0pt,0 pt) ++(-1.5pt,-1.5pt) -- ++(0 pt,3.0pt);
\draw [color=black] (7.0,0.0)-- ++(-1.5pt,0 pt) -- ++(3.0pt,0 pt) ++(-1.5pt,-1.5pt) -- ++(0 pt,3.0pt);
\draw [color=black] (0.0,1.0)-- ++(-1.5pt,0 pt) -- ++(3.0pt,0 pt) ++(-1.5pt,-1.5pt) -- ++(0 pt,3.0pt);
\draw [fill=black] (1.0,1.0) circle (2.5pt);
\draw [color=black] (2.0,1.0)-- ++(-1.5pt,0 pt) -- ++(3.0pt,0 pt) ++(-1.5pt,-1.5pt) -- ++(0 pt,3.0pt);
\draw [fill=black] (3.0,1.0) circle (2.5pt);
\draw [color=black] (4.0,1.0)-- ++(-1.5pt,0 pt) -- ++(3.0pt,0 pt) ++(-1.5pt,-1.5pt) -- ++(0 pt,3.0pt);
\draw [color=black] (5.0,1.0)-- ++(-1.5pt,0 pt) -- ++(3.0pt,0 pt) ++(-1.5pt,-1.5pt) -- ++(0 pt,3.0pt);
\draw [color=black] (6.0,1.0)-- ++(-1.5pt,0 pt) -- ++(3.0pt,0 pt) ++(-1.5pt,-1.5pt) -- ++(0 pt,3.0pt);
\draw [color=black] (7.0,1.0)-- ++(-1.5pt,0 pt) -- ++(3.0pt,0 pt) ++(-1.5pt,-1.5pt) -- ++(0 pt,3.0pt);
\draw [color=black] (0.0,2.0)-- ++(-1.5pt,0 pt) -- ++(3.0pt,0 pt) ++(-1.5pt,-1.5pt) -- ++(0 pt,3.0pt);
\draw [fill=black] (1.0,2.0) circle (2.5pt);
\draw [color=black] (2.0,2.0)-- ++(-1.5pt,0 pt) -- ++(3.0pt,0 pt) ++(-1.5pt,-1.5pt) -- ++(0 pt,3.0pt);
\draw [color=black] (3.0,2.0)-- ++(-1.5pt,0 pt) -- ++(3.0pt,0 pt) ++(-1.5pt,-1.5pt) -- ++(0 pt,3.0pt);
\draw [color=black] (4.0,2.0)-- ++(-1.5pt,0 pt) -- ++(3.0pt,0 pt) ++(-1.5pt,-1.5pt) -- ++(0 pt,3.0pt);
\draw [color=black] (5.0,2.0)-- ++(-1.5pt,0 pt) -- ++(3.0pt,0 pt) ++(-1.5pt,-1.5pt) -- ++(0 pt,3.0pt);
\draw [color=black] (6.0,2.0)-- ++(-1.5pt,0 pt) -- ++(3.0pt,0 pt) ++(-1.5pt,-1.5pt) -- ++(0 pt,3.0pt);
\draw [color=black] (7.0,2.0)-- ++(-1.5pt,0 pt) -- ++(3.0pt,0 pt) ++(-1.5pt,-1.5pt) -- ++(0 pt,3.0pt);
\draw [color=black] (0.0,3.0)-- ++(-1.5pt,0 pt) -- ++(3.0pt,0 pt) ++(-1.5pt,-1.5pt) -- ++(0 pt,3.0pt);
\draw [color=black] (1.0,3.0)-- ++(-1.5pt,0 pt) -- ++(3.0pt,0 pt) ++(-1.5pt,-1.5pt) -- ++(0 pt,3.0pt);
\draw [color=black] (2.0,3.0)-- ++(-1.5pt,0 pt) -- ++(3.0pt,0 pt) ++(-1.5pt,-1.5pt) -- ++(0 pt,3.0pt);
\draw [color=black] (3.0,3.0)-- ++(-1.5pt,0 pt) -- ++(3.0pt,0 pt) ++(-1.5pt,-1.5pt) -- ++(0 pt,3.0pt);
\draw [color=black] (4.0,3.0)-- ++(-1.5pt,0 pt) -- ++(3.0pt,0 pt) ++(-1.5pt,-1.5pt) -- ++(0 pt,3.0pt);
\draw [color=black] (5.0,3.0)-- ++(-1.5pt,0 pt) -- ++(3.0pt,0 pt) ++(-1.5pt,-1.5pt) -- ++(0 pt,3.0pt);
\draw [color=black] (6.0,3.0)-- ++(-1.5pt,0 pt) -- ++(3.0pt,0 pt) ++(-1.5pt,-1.5pt) -- ++(0 pt,3.0pt);
\draw [color=black] (7.0,3.0)-- ++(-1.5pt,0 pt) -- ++(3.0pt,0 pt) ++(-1.5pt,-1.5pt) -- ++(0 pt,3.0pt);
\end{tikzpicture}
\caption{Visualization of Theorem~\ref{thm:reformulation}}
\label{fig:2}
\end{figure}
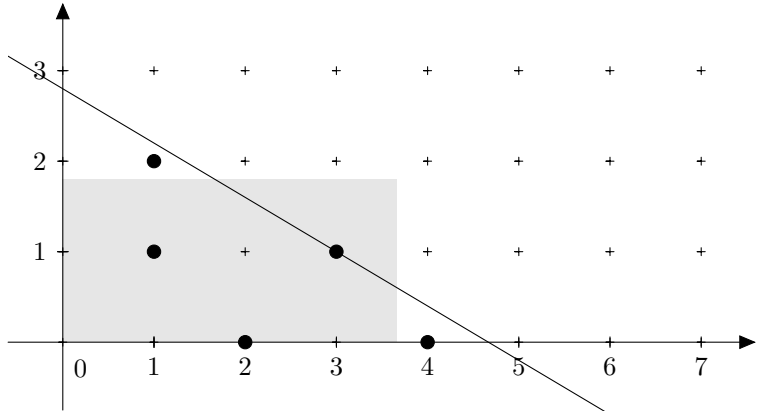
\end{ex}
In the case $k>2$ drawing a picture and fitting the right simplex might not be that easy. However, we can provide a further bound on $\kappa(f)$. Note that this bound is better than the bound given in Corollary~\ref{cor:dk} in the case that the degree of some {columns} $X_{\cdot j}$ is much smaller than the others.
\begin{cor}\label{cor:kk prod d}
	Let $k\geq 2$ and let $f\in\rx$ be a $k$-symmetric polynomial such that each column $X_{\cdot j}$ occurs only with degree  at most $d_j \geq 1$. Then
    \[\kappa(f) \leq k^k \prod_{j=1}^{k} d_j.\]
\end{cor}

\begin{proof}
Note that $E_f\subset[0,d_1]\times\dots\times[0,d_k] \subset \Delta(kd_1,\dots,kd_k)$ and use Theorem~\ref{thm:reformulation} {to deduce the claim}.
\end{proof}

\section{Deciding Convexity of multisymmetric polynomials}

In this section we apply {Theorem~\ref{thm:reformulation}} to the problem of algorithmically deciding convexity of a polynomial of fixed degree. Already for general quartic (i.e. degree 4) polynomials deciding convexity is an NP-hard problem (see \cite{amir}). However, we will see that in the case of symmetric (or, more generally, $k$-symmetric) polynomials of a bounded degree, a convexity test can be provided whose complexity is mainly determined by the degree in the sense that for a fixed degree it is polynomial in the number of variables. Recall that convexity of a function is defined as follows.

\begin{definition}
Let $C\subset\R^m$ be a convex set and $f:C\rightarrow \R$ be a real valued  function. Then $f$ is called \emph{convex} if $$\forall x_1, x_2 \in C, \forall t \in [0, 1]: f(tx_1+(1-t)x_2)\leq t f(x_1)+(1-t)f(x_2).$$

\end{definition}
We remark that convexity is of particular interest also for the question of deciding if a polynomial is non-negative.
\begin{prop}
Let $f\in\rx$. If $f$ is convex, then $\kappa(f)=1$.
\end{prop}
\begin{proof}
Let $\xi\in V^k$. Then 
$$f(\xi)=\sum_{\sigma\in S_n}\frac{1}{n!}f(\sigma(\xi))\geq f\left(\underbrace{\sum_{\sigma\in S_n}\frac{1}{n!}\sigma(\xi)}_{:=\zeta}\right).$$ Since  $\zeta\in A_1$, the statement follows.
\end{proof}
Using the classical definition above we remark the following characterizations which will allow the use of our main theorem.
\begin{prop}\label{prop:main}
Let $f:\R^{m} \to \R$ be a polynomial function.
\begin{enumerate}
\item Then $f$ is convex if and only if the polynomial function $g_f: \R^{2m} \to \R$,
\[g_f(x,\tilde x) := \tilde x^TD^2f(x)\tilde x = \sum_{i=1}^m\sum_{j=1}^m\tilde x_i \tilde x_j \partial_{ij}f(x)\]
is non-negative.
\item If $f\in \rx$ is a $k$-symmetric polynomial, then $g_f \in \R[V^k\oplus{V^k}]$ is a $2k$-symmetric polynomial.
\item Let $f\in \rx$ be a $k$-symmetric polynomial
and consider $\kappa(g_f)$ for the  $2k$-symmetric function $g_f$ defined in $(1)$. Then $f$ is convex if and only if the restriction
of $f$ onto each of the linear subspaces defining $A_{\kappa(g_f)}$ is convex. 
\end{enumerate}
\end{prop}
\begin{proof}
The first assertion is a direct consequence of Taylor's theorem. For the second statement observe that  for each $\sigma \in {S_m}$ such that $f(\sigma(x)) = f(x)$ for all $x\in\R^m$, we get $g_f(\sigma(x),\sigma(\tilde x)) = g_f(x,\tilde x)$ for all $x,\tilde x \in \R^m$, from which the assertion follows. The last statement follows directly from $(1)$ and $(2)$ and the definition of $\kappa(g_f)$.

\end{proof}

Now we fix $k$ and $d \geq 2$ and consider a $k$-symmetric polynomial $f$ of degree $d$. By Proposition~\ref{prop:main}, it follows that convexity of $f$ can algorithmically be checked using Theorem~\ref{thm:Main}.
 Let {$c(\ell,d)$ 
be an upper bound for the complexity of} deciding if a polynomial 
{in $\ell$ variables}
of degree at most $d$ is convex. Then, we can conclude with Propositions~\ref{prop:prop} {and~\ref{prop:numberPart}} that the complexity of deciding convexity of $f$ is bounded by {$c(k\cdot \kappa(f), d)\cdot n^{\kappa(g_f)}$}. Since $\deg (g_f) \leq d$, we can directly deduce from Corollary~\ref{cor:dk} that $\kappa(g_f) \leq d^{2k}$, which implies that the complexity {(for fixed $k$ and $d$)} grows  polynomially in the number of variables. Additionally, since for each of the resulting $k\cdot\kappa(f)$-dimensional polynomials convexity can be decided independently, this approach can be parallelized. 

However, the bound for $\kappa(g_f)$ can be improved considerably by exploiting the specific structure of $g_f$. 
Note that for any $x,\tilde{x}\in\R^{n\times k}$ the polynomial $g_f(X,\tilde{x})$ is of degree at most $d-2$ and $g_f(x,\tilde{X})$ of degree {at most} $2$.
So, Corollary\,$\ref{cor:kk prod d}$ can be used to infer that
\[\kappa(g_f) \leq 8^k k^{2k} (d-2)^k,\]
which is a better bound in case $d \geq 8k^2$. In fact, we will prove below {(Corollary~\ref{cor:conSym})} that the term $k^{2k}$ is superfluous. 

For this purpose, we examine the non-negativity of $g_f$ using the reformulation in terms of lattice points presented in Theorem~\ref{thm:reformulation}. In the following Proposition we examine the possible exponent vectors of $g_f$. We make use of Notation~\ref{not:morphismus,Mf,Ef} to formulate our result.
\begin{prop} \label{prop:Hf}
	Let $f\in\rx$ be a $k$-symmetric polynomial. For $i,j \in \{1,\dots,k\}$ we consider $\tau_{ij}: \R^k\to \R^k, x\mapsto x-e_i-e_j$ (where $e_1,\dots,e_k$ is the standard basis of $\R^k$). Then we have
    \begin{equation} \label{eq:EGF}
    E_{g_f} \subset \bigcup_{i=1}^k\bigcup_{j=1}^k (\tau_{ij}(E_f) \cap \R_{\geq 0}^k) \times \{e_i+e_j\}\subset \R^k\times \R^k.
    \end{equation}
In particular,
    \[E_{g_f} \subset H_f \times \Delta(2,\dots,2),\]
    where $H_f := \bigcup_{i,j} \tau_{ij}(E_f) \cap \R_{\geq 0}^k.$
\end{prop}

\begin{proof}
	This immediately follows by examining the definition of $g_f$: Just note that if $h$ is a second partial derivative of $f$ with respect to variables in columns $i$ and $j$, then $E_h \subset \tau_{ij}(E_f)\cap \R_{\geq 0}^k$.
\end{proof}
The above observations now yield the following Theorem.
\begin{thm}\label{thm:Hf}
	Let $f\in\rx$ be a $k$-symmetric polynomial and let $a_1,\dots,a_k\in\Q_{\geq 1}$ be such that $\Delta(a_1,\dots,a_k)$ is a simplex enclosing $H_f$ (as defined in Proposition~\ref{prop:Hf}). Then
    \[\kappa(g_f) \leq 3^k \prod_{j=1}^k \lfloor 2a_j\rfloor.\]
    In particular, $\kappa(g_f) \leq 6^k k!\,{\rm vol}(\Delta(a_1,\dots,a_k))$.
\end{thm}

\begin{proof}
	This follows directly from Theorem~\ref{thm:reformulation} with the preceding Proposition by noting that there exist $\varepsilon, \delta \in (0,1)$ such that $\Delta(a_1,\dots,a_k)\times\Delta(2,\dots,2)$ is contained in $\Delta(2a_1+\delta, \dots,2a_k+\delta,4-\varepsilon,\dots,4-\varepsilon)$ and such that $\lfloor 2a_j+\delta \rfloor = \lfloor 2a_j\rfloor$ for all $j\in\{1,\ldots,k\}$.
\end{proof}

\begin{ex}[Example~\ref{ex:1} continued]
Consider again the family of $2$-symmetric polynomials $f$ which was studied in Example~\ref{ex:1}. By reading $E_f$ from Figure~\ref{fig:1} we can easily construct $H_f$ (see Proposition~\ref{prop:Hf}). We need to fit a $w$-degree line such that the originating triangle encloses $H_f$. With a view to Theorem~\ref{thm:Hf}, we should choose a {triangle $\Delta(a_1,a_2)$}
that minimizes $\left\lfloor2a_1\right\rfloor\cdot\left\lfloor2a_2\right\rfloor$. With the choice indicated in
Figure~\ref{fig:3} {we get} $a_1<2.5$ and $a_2<2$. Hence, $\kappa(g_f)\leq 108$. However, it turns out that the bound can be improved by taking a closer look at $E_{g_f}$ using \eqref{eq:EGF}. By solving a small optimization problem we are able to find that  $E_{g_f}\subset\Delta(9-\varepsilon,3-\varepsilon,4-\varepsilon,2-\varepsilon)$ for a small $\varepsilon>0$. From this observation one can deduce the better bound $\kappa(g_f)\leq 48$.
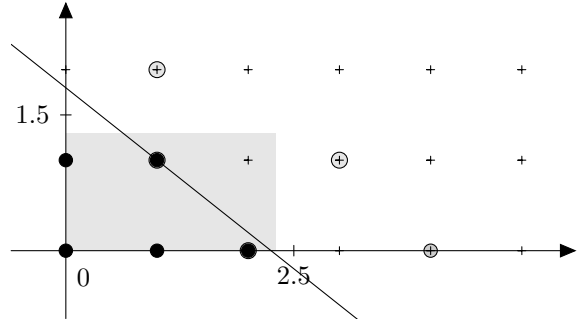
\begin{figure}
\begin{tikzpicture}[line cap=round,line join=round,>=triangle 45,x=1.2cm,y=1.2cm]
\draw[->,color=black] (-0.6,0.0) -- (5.6,0.0);

\draw[->,color=black] (0.0,-0.75) -- (0.0,2.75);

\draw[color=black] (0pt,-10pt) node[right] {\footnotesize $0$};
\draw[shift={(0,1.5)},color=black] (2pt,0pt) -- (-2pt,0pt) node[left] {\footnotesize $1.5$};
\draw[shift={(2.5,0)},color=black] (0pt,2pt) -- (0pt,-2pt) node[below] {\footnotesize $2.5$};
\clip(-0.6,-0.75) rectangle (5.6,2.75);

\draw [domain=-0.6:7.6] plot(\x,{-0.8*\x+1.8});

\fill[line width=0.0pt,fill=black,fill opacity=0.1] (0,1.3) -- (2.3,1.3) -- (2.3,0.0) -- (0.0,0.0) -- cycle;

\draw [color=black] (0.0,0.0)-- ++(-1.5pt,0 pt) -- ++(3.0pt,0 pt) ++(-1.5pt,-1.5pt) -- ++(0 pt,3.0pt);
\draw [color=black] (1.0,0.0)-- ++(-1.5pt,0 pt) -- ++(3.0pt,0 pt) ++(-1.5pt,-1.5pt) -- ++(0 pt,3.0pt);
\draw [color=black] (2.0,0.0)-- ++(-1.5pt,0 pt) -- ++(3.0pt,0 pt) ++(-1.5pt,-1.5pt) -- ++(0 pt,3.0pt);
\draw [fill=black,fill opacity=0.2] (2.0,0.0) circle (3pt);
\draw [color=black] (3.0,0.0)-- ++(-1.5pt,0 pt) -- ++(3.0pt,0 pt) ++(-1.5pt,-1.5pt) -- ++(0 pt,3.0pt);
\draw [color=black] (4.0,0.0)-- ++(-1.5pt,0 pt) -- ++(3.0pt,0 pt) ++(-1.5pt,-1.5pt) -- ++(0 pt,3.0pt);
\draw [fill=black,fill opacity=0.2] (4.0,0.0) circle (2.5pt);
\draw [color=black] (5.0,0.0)-- ++(-1.5pt,0 pt) -- ++(3.0pt,0 pt) ++(-1.5pt,-1.5pt) -- ++(0 pt,3.0pt);

\draw [color=black] (0.0,1.0)-- ++(-1.5pt,0 pt) -- ++(3.0pt,0 pt) ++(-1.5pt,-1.5pt) -- ++(0 pt,3.0pt);
\draw [color=black] (1.0,1.0)-- ++(-1.5pt,0 pt) -- ++(3.0pt,0 pt) ++(-1.5pt,-1.5pt) -- ++(0 pt,3.0pt);
\draw [fill=black,fill opacity=0.1] (1.0,1.0) circle (3pt);
\draw [color=black] (2.0,1.0)-- ++(-1.5pt,0 pt) -- ++(3.0pt,0 pt) ++(-1.5pt,-1.5pt) -- ++(0 pt,3.0pt);
\draw [color=black] (3.0,1.0)-- ++(-1.5pt,0 pt) -- ++(3.0pt,0 pt) ++(-1.5pt,-1.5pt) -- ++(0 pt,3.0pt);
\draw [fill=black,fill opacity=0.1] (3.0,1.0) circle (3pt);
\draw [color=black] (4.0,1.0)-- ++(-1.5pt,0 pt) -- ++(3.0pt,0 pt) ++(-1.5pt,-1.5pt) -- ++(0 pt,3.0pt);
\draw [color=black] (5.0,1.0)-- ++(-1.5pt,0 pt) -- ++(3.0pt,0 pt) ++(-1.5pt,-1.5pt) -- ++(0 pt,3.0pt);
\draw [color=black] (0.0,2.0)-- ++(-1.5pt,0 pt) -- ++(3.0pt,0 pt) ++(-1.5pt,-1.5pt) -- ++(0 pt,3.0pt);
\draw [color=black] (1.0,2.0)-- ++(-1.5pt,0 pt) -- ++(3.0pt,0 pt) ++(-1.5pt,-1.5pt) -- ++(0 pt,3.0pt);
\draw [fill=black,fill opacity=0.1] (1.0,2.0) circle (3pt);
\draw [color=black] (2.0,2.0)-- ++(-1.5pt,0 pt) -- ++(3.0pt,0 pt) ++(-1.5pt,-1.5pt) -- ++(0 pt,3.0pt);
\draw [color=black] (3.0,2.0)-- ++(-1.5pt,0 pt) -- ++(3.0pt,0 pt) ++(-1.5pt,-1.5pt) -- ++(0 pt,3.0pt);
\draw [color=black] (4.0,2.0)-- ++(-1.5pt,0 pt) -- ++(3.0pt,0 pt) ++(-1.5pt,-1.5pt) -- ++(0 pt,3.0pt);
\draw [color=black] (5.0,2.0)-- ++(-1.5pt,0 pt) -- ++(3.0pt,0 pt) ++(-1.5pt,-1.5pt) -- ++(0 pt,3.0pt);

\draw [fill=black] (1.0,1.0) circle (2.5pt);
\draw [fill=black] (0.0,0.0) circle (2.5pt);
\draw [fill=black] (2.0,0.0) circle (2.5pt);
\draw [fill=black] (1.0,0.0) circle (2.5pt);
\draw [fill=black] (0.0,1.0) circle (2.5pt);
\end{tikzpicture}
\caption{Visualization of Theorem~\ref{thm:Hf}}
\label{fig:3}
\end{figure}
\end{ex}Apart from this geometrical view we provide a formulation in terms of weighted degrees. Note that this formulation is a bit weaker than Theorem~\ref{thm:Hf}, but it might still be useful as we need less information about the function in question in order to calculate the resulting bound. 
\begin{thm}\label{thm:convex2}
	Let $w=(w_1,\dots,w_k)$ be a $k$-tuple of positive integers. Let {$d\geq 2\min_{1\leq j\leq k} w_j +\max_{1\leq j\leq k} w_j$} and let $f\in\rx$ be a $k$-symmetric polynomial of $w$-weighted degree at most $d$. Set $\tilde d := d-2{\min_{1\leq j\leq k} w_j}$. Then
    \[\kappa(g_f) \leq 3^k \prod_{j=1}^k \left\lfloor \frac{2\tilde d}{w_j} \right\rfloor.\]
\end{thm}

\begin{proof}
	Since $f$ is of $w$-weighted degree at most $d$, we have $E_f\subset \Delta(d/w_1,\dots,d/w_k)$. Hence, by definition of $\tilde{d}$, we have $H_f \subset \Delta(\tilde d/w_1, \dots, \tilde d/w_k)$. Finally, note that $\tilde{d}/w_j\geq 1$, since {$d\geq 2\min_{1\leq j\leq k} w_j+\max_{1\leq j\leq k} w_j$}.
\end{proof}
Choosing all weights equal to $1$ we get the following Corollary, 
\begin{cor} \label{cor:conSym}
	Let $f\in\rx$ be a $k$-symmetric polynomial of degree $d \geq 3$. Then
    \[\kappa(g_f) \leq 6^k (d-2)^k.\]
\end{cor}
The following Corollary is  just a reformulation of our results, which we include to emphasize the results in the case of convexity of symmetric polynomials.
\begin{cor}
	Let $f\in\R[V]^{S_n}$ be a symmetric polynomial of degree $d \geq 3$. Then $f$ is a convex function if and only if it is convex on each of the subspaces of points with at most $6(d-2)$ distinct {coordinates}. 

\end{cor}

Finally, we shortly illustrate our results with the following example.

\begin{ex}

We consider the symmetric polynomial

\[
f(X_1,\ldots,X_{25})=
-\frac{1}{2}\sum_{i=1}^{25}  X_i^4 +
\sum_{i,j=1}^{25} X_i^2X_j^2 +
\frac{1}{2}\sum_{i,j,k,l=1}^{25} X_iX_jX_kX_l +
\sum_{i=1}^{25} X_i^2.
\]

It can now be deduced from Proposition~\ref{prop:main} that in order to verify that $f$  defines a convex function it suffices to check non-negativity of the 2-symmetric polynomial $g_f\in\R[X_1, \tilde X_1,\ldots,X_{25},\tilde{X}_{25}]$, given by

\begin{align*}
g_f = 
\sum_{k=1}^{25}\tilde{X}_k^2 \left(-6X_k^2 + 4\sum_{i=1}^{25} X_i^2 +2 \right) +
\sum_{k,\ell=1}^{25} \tilde{X}_k\tilde{X}_\ell \left(8X_kX_\ell + 6\sum_{i,j=1}^{25}X_iX_j \right).
\end{align*}

	Corollary~\ref{cor:conSym} allows us to reduce the problem size: Indeed, we find $\kappa(g_f) \leq 12$, and hence it is sufficient   to test non-negativity of $g_f(X,\tilde X)$ on all points of the form

\begin{eqnarray*} 
\begin{array}{c@{\!\!\!}l}
  \left(
   \begin{array}[c]{cc}
    x_1    & \tilde{x}_1 \\
    \vdots & \vdots   \\
        &   \\
    \hline
        &   \\
    \vdots & \vdots   \\
        &   \\
    \hline
    \vdots & \vdots   \\
    \hline
        &   \\
    \vdots & \vdots   \\
    x_{25}    & \tilde{x}_{25}
   \end{array}  
  \right) 
\end{array} 
&=
\begin{array}{c@{\!\!\!}l}
  \left(
   \begin{array}[c]{cc}
    y_1    & \tilde{y}_1 \\
    \vdots & \vdots   \\
    y_1    & \tilde{y}_1 \\
    \hline
    y_2    & \tilde{y}_2 \\
    \vdots & \vdots   \\
    y_2    & \tilde{y}_2 \\
    \hline
    \vdots & \vdots   \\
    \hline
    y_{12}    & \tilde{y}_{12} \\
    \vdots & \vdots   \\
    y_{12}    & \tilde{y}_{12}
   \end{array}  
  \right) 
\end{array} 
&
\begin{array}[c]{@{}l@{\,}l}
  \left. 
   \begin{array}{c} 
    \vphantom{0}  \\ \vphantom{} \\ \vphantom{0} 
   \end{array} 
  \right\} & \text{$\lambda_1$ times} \\
  \left. 
   \begin{array}{c} 
    \vphantom{0} \\ \vphantom{\vdots} \\ \vphantom{0}  
   \end{array} 
  \right\} & \text{$\lambda_2$ times} \\
 \vphantom{\vdots} & \\ 
  \left. 
   \begin{array}{c} 
    \vphantom{0}  \\ \vphantom{\vdots} \\ \vphantom{0} 
   \end{array} 
  \right\} & \text{$\lambda_{12}$ times}
\end{array} 
\end{eqnarray*}

This in turn implies that  $f(X_1,\dots,X_{25})$ is convex if and only if, for all ${\lambda \vdash_{12} 25}$, the polynomial $g_f^{(\lambda)}\in\R[Y_1,\tilde{Y}_1,\ldots,Y_{12},\tilde{Y}_{12}]$ resulting from the above described substitution, namely

\begin{equation*}
g_f^{(\lambda)}=\sum_{k=1}^{12}\lambda_k\tilde{Y}_k^2 \left( -6Y_k^2 + 4\sum_{i=1}^{12} \lambda_iY_i^2+2 \right) +
\sum_{k,\ell=1}^{12} \lambda_k\tilde{Y}_k\lambda_\ell\tilde{Y}_\ell \left(8Y_k Y_\ell + 6\sum_{i,j=1}^{12}\lambda_iY_i\lambda_jY_j \right)
\end{equation*}
is non-negative.

Now observe that the  number of $12$-partitions of $25$ is $100$. 
Hence the various substitutions described above produce 100  polynomials each of which only involves $24$ variables. Thus the problem of testing non-negativity of $g_f$, which is a polynomial in 50 variables 
reduces to checking if all of the $24$-variate polynomials are non-negative. Whereas the authors were not directly able to verify the non-negativity of $g_f$, it was possible to numerically verify the non-negativity of the $100$ resulting polynomials in fewer variables using a standard numerical sums-of-squares implementation. 

\end{ex}

\section*{Acknowledgement}
The research presented in this article was carried out during the first and third author's stay at the Aalto Science Institute during the summer internship program 2014. The authors would like to thank the Aalto Science Institute for the hospitality and support. Further, we would like to thank two anonymous referees for insightful comments.  \nocite{*}
\bibliographystyle{abbrv}

\end{document}